\documentclass[preprint,11pt]{elsarticle}
\topskip  \usepackage{amssymb}
\usepackage{amsmath}
\usepackage{amsthm}
\usepackage{hyperref}
\usepackage{rotating}
\usepackage[enable-survey]{pdfpages}
\newtheorem{definition}{Definition}
\newtheorem{theorem}{Theorem}

\newtheorem{corollary}{Corollary}
\newtheorem{remark}{Remark}
\numberwithin{equation}{section} 
\numberwithin{definition}{section}
\numberwithin{theorem}{section}
\numberwithin{lemma}{section}
\numberwithin{corollary}{section}
\numberwithin{remark}{section}

\begin{document}
\title{\bf{New Extension of Unified Family of Apostol-Type of Polynomials and Numbers}}
\author[Desouky]{B. S. El-Desouky \corref{cor2}}
\author[Gomaa]{R. S. Gomaa}
\address{${ }^{1}$Department of Mathematics, Faculty of Science, Mansoura University, 35516 Mansoura, Egypt}

\cortext[cor2]{Corresponding author}\ead{b\_desouky@yahoo.com {(B. S. Desouky)}}
\begin{abstract}
The purpose of this paper is to introduce and investigate a new unification of unified family of Apostol-type polynomials and numbers based on results given in \cite{Sriva2} and \cite{Sriva3}. Also, we derive some properties for these polynomials and obtain some relationships between the Jacobi polynomials, Laguerre polynomials, Hermite polynomials, Stirling numbers and some other types of generalized polynomials.  

\noindent {\bf Key words}: Generalized Euler, Bernoulli and Genocchi polynomials; Stirling numbers; generalized Stirling numbers; Laguerre polynomials; Hermite polynomials; Jacobi polynomials.

\noindent {\bf AMS Subject Classification}: 05A10, 11B68, 11B73, 11B83, 11M06, 11M35, 33E20
\end{abstract}
\maketitle
\section{Introduction} 
The generalized Bernoulli polynomials $B^{(\alpha)}_{n}(x)$ of order $\alpha \in \mathbb{C}$ and the generalized Euler polynomials are defined by (see \cite{Sriva1}):
\begin{equation}\label{eq1}
\left(\frac{t}{e^{t}-1}\right)^{\alpha}e^{xt}=\sum\limits_{n=0}^{\infty}B^{(\alpha)}_{n}(x)\frac{t^{n}}{n!},\text{  } \left(\vert t \vert <2\pi; 1^{\alpha}:=1 \right)
\end{equation}
and
\begin{equation}\label{eq2}
\left(\frac{t}{e^{t}+1}\right)^{\alpha}e^{xt}=\sum\limits_{n=0}^{\infty}E^{(\alpha)}_{n}(x)\frac{t^{n}}{n!},  \text{  } \left(\vert t \vert <\pi; 1^{\alpha}:=1\right),
\end{equation}
where $\mathbb{C}$ denote set of complex numbers.\\
Recently, Luo and Srivastava \cite{Luo3} introduced the generalized Apostol-Bernoulli polynomials $B^{(\alpha)}_{n}(x;\lambda)$ and the generalized Apostol-Euler polynomials $E^{(\alpha)}_{n}(x;\lambda)$ as follows:
\begin{definition} (Luo and Srivastava \cite{Luo3}) The generalized Apostol-Bernoulli polynomials $B^{(\alpha)}_{n}(x;\lambda)$ of order $\alpha \in \mathbb{C}$ are defined by the generating function
$$
\left(\frac{t}{\lambda e^{t}-1}\right)^{\alpha}e^{xt}=\sum\limits_{n=0}^{\infty}B^{(\alpha)}_{n}(x;\lambda)\frac{t^{n}}{n!}$$
\begin{equation}\label{eq3}
 \left(\vert t \vert <2\pi \text{ when } \lambda=1; \vert t \vert < \vert \log\lambda \vert,\text{ when } \lambda \neq 1; 1^{\alpha}:=1 \right).
\end{equation}
\end{definition}
\begin{definition} (Luo \cite{Luo4}) The generalized Apostol-Euler polynomials $E^{(\alpha)}_{n}(x;\lambda)$ of order $\alpha \in \mathbb{C}$ are defined by the generating function
$$
\left(\frac{t}{\lambda e^{t}+1}\right)^{\alpha}e^{xt}=\sum\limits_{n=0}^{\infty}E^{(\alpha)}_{n}(x;\lambda)\frac{t^{n}}{n!}$$
\begin{equation}\label{eq4}
 \left(\vert t \vert <\pi \text{ when } \lambda=1; \vert t \vert < \vert \log(-\lambda) \vert,\text{ when } \lambda \neq 1; 1^{\alpha}:=1 \right).
\end{equation}
\end{definition}
Natalini and Bernardini \cite{Nat} defined the new generalization of Bernoulli polynomials in the following form.
\begin{definition} The generalized Bernoulli polynomials $ B^{[m-1]}_{n}(x)$, $m \in \mathbb N $, are defined, in a suitable neighbourhood of $t=0$ by means of generating function
\begin{equation}\label{eq5}
 \frac{t^{m}e^{xt}}{e^{t}-\sum\limits_{l=0}^{m-1}\frac{t^{l}}{l!}} =\sum\limits_{n=0}^{\infty}B^{[m-1]}_{n}(x)\frac{t^{n}}{n!}.
\end{equation}
\end{definition}
Recently, Tremblay et al. \cite{Tremblay} investigated a new class of generalized Apostol-Bernoulli polynomials. These are defined as follows.
\begin{definition} The generalized Apostol-Bernoulli polynomials $ B^{[m-1,\alpha]}_{n}(x;\lambda)$ of order $\alpha \in \mathbb{C}$, $m \in \mathbb N $, are defined, in a suitable neighbourhood of $t=0$ by means of generating function
\begin{equation}\label{eq6}
\left( \frac{t^{m}}{\lambda e^{t}-\sum\limits_{l=0}^{m-1}\frac{t^{l}}{l!}}\right)^{\alpha}e^{xt}=\sum\limits_{n=0}^{\infty}B^{[m-1,\alpha]}_{n}(x;\lambda)\frac{t^{n}}{n!}.
\end{equation}
\end{definition} 
Also, Sirvastava et al. \cite{Sriva2} introduced a new interesting class of Apostol-Bernoulli polynomials that are closely related to the new class that we present in this paper. They investigated the following form.
\begin{definition} Let $ a,b,c \in \mathbb{R^{+}}(a \neq b))$ and $n \in \mathbb{N}_{0} $. Then the generalized Bernoulli polynomials $\mathfrak{B}^{(\alpha)}_{n}(x;\lambda;a,b,c)$ of order $\alpha \in \mathbb{C}$ are defined by the following generating function:
$$
\left(\frac{t}{\lambda b^{t}-a^{t}}\right)^{\alpha}c^{xt}=\sum\limits_{n=0}^{\infty}\mathfrak{B}^{(\alpha)}_{n}(x;\lambda;a,b,c)\frac{t^{n}}{n!}$$
\begin{equation}\label{eq7}
 \left( \left \vert t\log\left(\frac{a}{b}\right) \right \vert < \vert \log\lambda \vert ;\text{ } 1^{\alpha}:=1 \right).
\end{equation}
\end{definition}
 In this sequel to the work by Sirvastava et al. \cite{Sriva3} introduced and investigated a similar generalization of the family of Euler polynomials defined as follws.
\begin{definition} Let $ a,b,c \in \mathbb{R^{+}}(a \neq b))$ and $n \in \mathbb{N}_{0} $. Then the generalized Euler polynomials $\mathfrak{E}^{(\alpha)}_{n}(x;\lambda;a,b,c)$ of order $\alpha \in \mathbb{C}$ are defined by the following generating function:
$$
\left(\frac{t}{\lambda b^{t}+a^{t}}\right)^{\alpha}c^{xt}=\sum\limits_{n=0}^{\infty}\mathfrak{E}^{(\alpha)}_{n}(x;\lambda;a,b,c)\frac{t^{n}}{n!}$$
\begin{equation}\label{eq8}
 \left( \left \vert t\log\left(\frac{a}{b}\right) \right \vert <\vert \log(-\lambda ) \vert ; \text{ }1^{\alpha}:=1 \right).
\end{equation}
\end{definition}
It is easy to see that setting $a=1$ and $b=c=e$ in \eqref{eq8} would lead to Apostol-Euler polynomials defined by \eqref{eq4}. The case where $\alpha=1$ has been studied by Luo et al. \cite{Luo1}.\\

In Section 2, we introduce the new extension of unified family of Apostol-type polynomials and numbers that are defined in \cite{Desouky2}. Also, we determine relation between some results given in \cite{Sriva1, Sriva2, Kurt1, Kurt2, Tremblay} and our results and introduce some new identities for polynomials defined in \cite{Desouky2} . In Section 3, we give some basic properties of   
the new unification of Apostol-type polynomials and numbers. Finally in Section 4, we introduce some relationships between the new unification of  Apostol-type polynomials and other known polynomials.
\section{ Unification of multiparameter Apostol-type polynomials and numbers}
\begin{definition} Let $ a,b,c \in \mathbb{R^{+}}(a \neq b))$, $n \in \mathbb{N}_{0} $ and $m \in \mathbb N $. Then the new unification of Apostol-type polynomials $M^{[m-1,r]}_{n}(x;k;a,b,c;\overline{\alpha}_{r})$ are defined, in a suitable neighbourhood of $t=0$ by means of generating function\\
$$F^{[m-1,r]}_{\overline{\alpha}_{r}}=\frac{t^{rkm}2^{rm(1-k)}c^{xt}}{\prod\limits_{i=0}^{r-1}\left(\alpha_{i}b^{t}-a^{t}\sum\limits_{\ell=0}^{m-1}\frac{t^{\ell}}{\ell!}\right)}=\sum\limits_{n=0}^{\infty} M^{[m-1,r]}_{n}(x;k;a,b,c;\overline{\alpha}_{r})\frac{t^{n}}{n!}$$
\\
$$
\Big( \left \vert t\log \left( \frac{b}{a}\right) \right \vert < 2\pi \text{ when }m=1 \text{ and } \alpha_{i} =1 ;\text{ } \left \vert t \log \left( \frac{b}{a}\right) \right \vert < \vert \log(\alpha_{i}) \vert 
$$
\begin{equation}\label{eq9}
\text{ when } m=1 \text{ and } \alpha_{i} \neq 1; \forall\; i=0,1,...,r-1 \Big), 
\end{equation}
where $k\in N_{0};r\in \mathbb{C};\overline{\alpha}_{r}=(\alpha_{0},\alpha
_{1},...,\alpha_{r-1}) \text{ is a sequence of complex numbers} \nonumber.$
\end{definition}

\begin{sidewaystable}[!htbp]
The generating function in \eqref{eq9} gives many types of polynomials as special cases, for example, see the following table
\caption{}
\begin{tabular}{|l|l l|}
\hline\hline
1&{\footnotesize setting  $k=1,\alpha_{i}=\lambda,$ $ i=0,1,...,r-1$, hence if $m=1$ in \eqref{eq9}} & \begin{tabular}{      l rr} {\footnotesize $M^{[0,r]}_{n}(x;1;a,b,c;\lambda)=\mathfrak{B}^{(r)}_{n}(x;\lambda;a,b,c)$ }& \\  ({\footnotesize generalized Bernoulli polynomials of order $r$, see \cite{Sriva3}})\end{tabular}
\\\hline
2&{\footnotesize setting  $k=0,\alpha_{i}=-\lambda,$ $ i=0,1,...,r-1$, hence if $m=1$ in \eqref{eq9}} & \begin{tabular}{l rr} {\footnotesize $M^{[0,r]}_{n}(x;0;a,b,c;-\lambda)=(-1)^{r}\mathfrak{E}^{(r)}_{n}(x;\lambda;a,b,c)$ }& \\  ({\footnotesize generalized Euler polynomials of order $r$, see \cite{Sriva3}})\end{tabular}
\\\hline
3&{\footnotesize setting  $\alpha_{i}=\beta,$ $ i=0,1,...,r-1,c=b$, hence if $m=1$ in \eqref{eq9}} & \begin{tabular}{ l rr} {\footnotesize $M^{[0,r]}_{n}(x;k;a,b,b;\beta)=y^{(r)}_{n,\beta}(x;k;a,b)$ }& \\  ({\footnotesize unification of Apostol-type polynomials of order $r$, see \cite{Ozden2}})\end{tabular}
\\\hline
4&\begin{tabular}{l rr}{\footnotesize setting  $k=1,t=t \ln a,x=\frac{x}{\ln a},\alpha_{i}=\lambda,$ $ i=0,1,...,r-1$},&\\{\footnotesize hence if $a=1,b=c^{\frac{1}{\ln a}}$ in \eqref{eq9}} \end{tabular} & \begin{tabular}{l rr } {\footnotesize $M^{[m-1,r]}_{n}(\frac{x}{\ln a};1,1,c^{\frac{1}{\ln a}},c;\lambda)=(\ln a)^{mr}B^{[m-1,r]}_{n}(x;c,a;\lambda)$ }& \\  ({\footnotesize generalized Bernoulli polynomials of order $r$, see \cite{Kurt2}})\end{tabular}
\\\hline
5&\begin{tabular}{l rr}{\footnotesize setting  $k=0,t=t \ln a, x=\frac{x}{\ln a},\alpha_{i}=-\lambda,$ $ i=0,1,...,r-1$},&\\{\footnotesize hence if $a=1,b=c^{\frac{1}{\ln a}}$ in \eqref{eq9}} \end{tabular} & \begin{tabular}{ l rr} {\footnotesize $M^{[m-1,r]}_{n}(\frac{x}{\ln a};0;1,c^{\frac{1}{\ln a}},c;-\lambda)=(-1)^{r}(\ln a)^{mr}E^{[m-1,r]}_{n}(x;c,a;\lambda)$ }& \\  ({\footnotesize generalized Euler polynomials of order $r$, see \cite{Kurt2}})\end{tabular}
\\\hline
6&\begin{tabular}{l rr}{\footnotesize setting  $k=1,\alpha_{i}=1,$ $ i=0,1,...,r-1,a=1,b=e,c=e$},&\\{\footnotesize hence if $r=1$ in \eqref{eq9}} \end{tabular} & {\footnotesize $M^{[m-1,1]}_{n}(x;1;1,e,e,1)=B^{[m-1]}_{n}(x)$ ( generalized Bernoulli polynomials, see \cite{Nat}})
\\\hline
7&\begin{tabular}{l rr}{\footnotesize setting $k=0,\alpha_{i}=-1,$ $ i=0,1,...,r-1,a=1,b=e,c=e$},&\\{\footnotesize hence if $r=1$ in \eqref{eq9}} \end{tabular} & {\footnotesize $M^{[m-1,1]}_{n}(x;0;1,e,e;-1)=-E^{[m-1]}_{n}(x)$ ( generalized Euler polynomials, see \cite{Nat}})
\\\hline
8& {\footnotesize setting $k=1,\alpha_{i}=1,$ $ i=0,1,...,r-1,a=1,b=e,c=e$ in \eqref{eq9}} & \begin{tabular}{l rr}{\footnotesize $M^{[m-1,r]}_{n}(x;1;1,e,e;1)=B^{[m-1,r]}_{n}(x)$ }&\\ {\footnotesize( generalized Bernoulli polynomials of order $r$, see \cite{Kurt1}}) \end{tabular}
\\\hline
9& {\footnotesize setting $k=0,\alpha_{i}=-1,$ $ i=0,1,...,r-1,a=1,b=e,c=e$ in \eqref{eq9}} & \begin{tabular}{l rr}{\footnotesize $M^{[m-1,r]}_{n}(x;0;1,e,e;-1)=(-1)^{r}E^{[m-1,r]}_{n}(x)$} &\\ {\footnotesize( generalized Euler polynomials of order $r$, see \cite{Kurt1}}) \end{tabular}
\\\hline
10& {\footnotesize setting $k=1,\alpha_{i}=-1,$ $ i=0,1,...,r-1,a=1,b=e,c=e$ in \eqref{eq9}} & \begin{tabular}{l rr}{\footnotesize $M^{[m-1,r]}_{n}(x;1;1,e,e;-1)=(-1)^{r}\left(\frac{1}{2}\right)^{rm}G^{[m-1,r]}_{n}(x)$} &\\ {\footnotesize( generalized Genocchi polynomials of order $r$, see \cite{Kurt1}}) \end{tabular}
\\\hline
11& {\footnotesize setting $k=1,\alpha_{i}=\lambda,$ $ i=0,1,...,r-1,a=1,b=e,c=e$ in \eqref{eq9}} & \begin{tabular}{l rr}{\footnotesize $M^{[m-1,r]}_{n}(x;1;1,e,e;\lambda)=B^{[m-1,r]}_{n}(x;\lambda)$ }&\\ {\footnotesize( generalized Apostol-Bernoulli polynomials of order $r$, see \cite{Tremblay}}) \end{tabular}
\\\hline
12& {\footnotesize setting $k=0,\alpha_{i}=-\lambda,$ $ i=0,1,...,r-1,a=1,b=e,c=e$ in \eqref{eq9}} & \begin{tabular}{l rr}{\footnotesize $M^{[m-1,r]}_{n}(x;0;1,e,e;-\lambda)=(-1)^{r}E^{[m-1,r]}_{n}(x;\lambda)$} &\\ {\footnotesize( generalized Apostol-Euler polynomials of order $r$, see \cite{Tremblay}}) \end{tabular}
\\\hline
13& {\footnotesize setting $m=1,a=1,b=e,c=e$ in \eqref{eq9}}&\begin{tabular}{l rr} { \footnotesize $M^{[0,r]}_{n}(x;k;1,e,e;-\overline{\alpha}_{r})=M^{(r)}_{n}(x;k;\overline{\alpha}_{r})$}&\\
{\footnotesize (a new unified family of generalized Apostol-Euler, Bernoulli}&\\
{\footnotesize and Genocchi polynomials, see \cite{Desouky2}})
\end{tabular}
\\\hline\hline
\end{tabular}
\end{sidewaystable}
\newpage
\begin{remark} If we set $x=0$ in \eqref{eq9}, then we obtain the new unification of multiparameter Apostol-type numbers, as 
\begin{equation}\label{eq10}
M^{[m-1,r]}_{n}(0;k;a,b,c;\overline{\alpha}_{r})=M^{[m-1,r]}_{n}(k;a,b,c;\overline{\alpha}_{r}).
\end{equation}
\end{remark}
\begin{remark} From $\textbf{No.13}$ in $\textbf{Table1}$ and \cite[\textbf{Table1}]{Desouky2}, we can obtain the polynomials and the numbers given in \cite{Apostol, Dere, Karande, Luo2, Ozden2}. 
\end{remark}
\section{ ٍSome basic properties for the polynomial $ M^{[m-1,r]}_{n}(x;k;a,b,c;\overline{\alpha}_{r})$}
\begin{theorem} Let $ a,b,c \in \mathbb{R^{+}}(a \neq b))$ and $x \mathbb \in {R} $. Then
{\footnotesize  \begin{eqnarray}
 M^{[m-1,r]}_{n}(x+y;k;a,b,c;\overline{\alpha}_{r})&=& \sum\limits_{l=0}^{n}\binom{n}{l}x^{n-l}(\ln c)^{n-l}M^{[m-1,r]}_{l}(y;k;a,b,c;\overline{\alpha}_{r}). \label{eq11}\\
M^{[m-1,r]}_{n}(x+r;k;a,b,c;\overline{\alpha}_{r})&=&M^{[m-1,r]}_{n}\left(x;k;\frac{a}{c},\frac{b}{c},c;\overline{\alpha}_{r}\right). \label{eq12}
\end{eqnarray}}
\end{theorem}
\begin{proof} For the first equation, from \eqref{eq9}
 \begin{eqnarray*}
 \sum\limits_{n=0}^{\infty} M^{[m-1,r]}_{n}(x+y;k;a,b,c;\overline{\alpha}_{r})\frac{t^{n}}{n!}&=&\frac{t^{rkm}2^{rm(1-k)}c^{xt}}{\prod\limits_{i=0}^{r-1}\left(\alpha_{i}b^{t}-a^{t}\sum\limits_{\ell=0}^{m-1}\frac{t^{\ell}}{\ell!}\right)}c^{yt}\\
& =&\sum\limits_{j=0}^{\infty}\frac{(ty \ln c)^{j}}{j!}\sum\limits_{l=0}^{\infty}M^{[m-1,r]}_{l}(x;k;a,b,c;\overline{\alpha}_{r})\frac{t^{l}}{l!},
 \end{eqnarray*}
   using Cauchy product rule, we can easily obtain \eqref{eq11}.\\
   For the second equation \eqref{eq12}, from \eqref{eq9}
  \begin{eqnarray*}
 \sum\limits_{n=0}^{\infty} M^{[m-1,r]}_{n}(x+r;k;a,b,c;\overline{\alpha}_{r})\frac{t^{n}}{n!}&=&\frac{t^{rkm}2^{rm(1-k)}}{\prod\limits_{i=0}^{r-1}\left(\alpha_{i}\left(\frac{b}{c}\right)^{t}-\left(\frac{a}{c}\right)^{t}\sum\limits_{\ell=0}^{m-1}\frac{t^{\ell}}{\ell!}\right)}c^{xt}\\
& =&\sum\limits_{n=0}^{\infty} M^{[m-1,r]}_{n}(x;k;\frac{a}{c},\frac{b}{c},c;\overline{\alpha}_{r})\frac{t^{n}}{n!}.
 \end{eqnarray*} 
 Equating coefficient of $\frac{t^{n}}{n!}$ on both sides, yields \eqref{eq12}. 
 \end{proof}  
 \begin{corollary} If $y=0$ in \eqref{eq11}, we have
  \begin{eqnarray}
 M^{[m-1,r]}_{n}(x;k;a,b,c;\overline{\alpha}_{r})&=& \sum\limits_{\ell=0}^{n}\binom{n}{\ell}x^{n-\ell}(\ln c)^{n-\ell} M^{[m-1,r]}_{\ell}(k;a,b,c;\overline{\alpha}_{r}) \label{eq13}\\
 &=& \sum\limits_{\ell=0}^{n}\binom{n}{n-\ell}x^{\ell}(\ln c)^{\ell} M^{[m-1,r]}_{n-\ell}(k;a,b,c;\overline{\alpha}_{r}) \label{eq14}.
 \end{eqnarray}
 \end{corollary}
 \begin{theorem} The following identity holds true, when $m=1$ and $\alpha_{i} \neq 0$ in \eqref{eq9} $ \text{  }\forall i=0,1,...,r-1$
{\footnotesize \begin{equation}\label{eq15}
 M^{[0,r]}_{n}(r-x;k;a,b,c;\overline{\alpha}_{r})=\frac{(-1)^{r(1-k)+n}}{\prod\limits_{i=0}^{r-1}\alpha_{i}}\sum\limits_{m=0}^{n}\binom{n}{m}\left(r \ln\left(\frac{ab}{c}\right)\right)^{n-m} M^{[0,r]}_{m}\left(x;k;a,b,c;\frac{1}{\overline{\alpha}_{r}}\right).
 \end{equation}}
 \end{theorem}
 \begin{proof} From \eqref{eq9}
{\footnotesize \begin{align*}
\sum\limits_{n=0}^{\infty} M^{[0,r]}_{n}(r-x;k;a,b,c;\overline{\alpha}_{r})\frac{t^{n}}{n!}& =\frac{t^{rk}2^{r(1-k)}c^{(r-x)t}}{\prod\limits_{i=0}^{r-1}\left(\alpha_{i}b^{t}-a^{t}\right)}\\
& =\frac{(-1)^{r(1-k)}}{\left(b^{t}a^{t}\right)^{r}\prod\limits_{j=0}^{r-1}\alpha_{j}}\frac{(-t)^{rk}2^{r(1-k)}c^{-xt}}{\prod\limits_{i=0}^{r-1}\left(\frac{b^{-t}}{\alpha_{i}}-a^{-t}\right)}c^{rt}\\
& =\frac{(-1)^{r(1-k)}}{\prod\limits_{j=0}^{r-1}\alpha_{j}}\left(\frac{ba}{c}\right)^{-rt}\sum\limits_{m=0}^{\infty} M^{[0,r]}_{m}\left(x;k;a,b,c;\frac{1}{\overline{\alpha}_{r}}\right)\frac{(-t)^{m}}{m!}\\
& =\frac{(-1)^{r(1-k)}}{\prod\limits_{j=0}^{r-1}\alpha_{j}}\sum\limits_{\ell=0}^{\infty}\frac{\left(r \ln\left(\frac{ab}{c}\right)\right)^{\ell}}{\ell!}(-t)^{\ell}\sum\limits_{m=0}^{\infty} M^{[0,r]}_{m}\left(x;k;a,b,c;\frac{1}{\overline{\alpha}_{r}}\right)\frac{(-t)^{m}}{m!}.
\end{align*}}
 Hence, we can easily obtain \eqref{eq15}.
\end{proof}
\begin{remark} If we put $\alpha_{i}=\beta,i=0,1,...,r-1$, $c=b$ and $r=\upsilon$ in \eqref{eq15}, then it gives \cite[Eq. (34)]{Ozden2},
\begin{equation*}
 M^{[0,\upsilon]}_{n}(\upsilon-x;k;a,b,b;\beta)=\frac{(-1)^{\upsilon(1-k)+n}}{(\beta)^{\upsilon}}\sum\limits_{m=0}^{n}\binom{n}{m}\left((\upsilon \ln a\right)^{n-m} M^{[0,\upsilon]}_{m}\left(x;k;a,b,b;\beta^{-1}\right),
 \end{equation*}
 where $M^{[0,\upsilon]}_{m}\left(x;k;a,b,b;\beta^{-1}\right)$ is the unification of the Apostol-type polynomials.
\end{remark}
\begin{theorem} The unification of Apostol-type numbers satisfy
\begin{equation}\label{eq16}
M^{[m-1,r]}_{n}(k;a,b,c;\overline{\alpha}_{r})=\sum\limits_{l=0}^{n}\binom{n}{l} M^{[m-1,\ell]}_{l}(k;a,b,c;\overline{\alpha}_{\ell}) M^{[m-1,r-\ell]}_{n-l}(k;a,b,c;\overline{\alpha}_{r-\ell}).
\end{equation}
\end{theorem}
\begin{proof} When $x=0$ in \eqref{eq9}, we have
{\footnotesize \begin{eqnarray*}
\sum\limits_{n=0}^{\infty} M^{[m-1,r]}_{n}(k;a,b,c;\overline{\alpha}_{r})\frac{t^{n}}{n!}&=&\frac{t^{rkm}2^{rm(1-k)}}{\prod\limits_{i=0}^{r-1}\left(\alpha_{i}b^{t}-a^{t}\sum\limits_{\ell=0}^{m-1}\frac{t^{\ell}}{\ell!}\right)}\\
&=&\frac{t^{\ell km}2^{\ell m(1-k)}}{\prod\limits_{i=0}^{\ell-1}\left(\alpha_{i}b^{t}-a^{t}\sum\limits_{\ell=0}^{m-1}\frac{t^{\ell}}{\ell!}\right)}\frac{t^{(r-\ell)km}2^{(r-\ell)m(1-k)}}{\prod\limits_{i=\ell}^{r-1}\left(\alpha_{i}b^{t}-a^{t}\sum\limits_{\ell=0}^{m-1}\frac{t^{\ell}}{\ell!}\right)}\\
&=&\sum\limits_{\ell_{1}=0}^{\infty} M^{[m-1,\ell]}_{\ell_{1}}(k;a,b,c;\overline{\alpha}_{\ell})\frac{t^{\ell_{1}}}{\ell_{1}!}\sum\limits_{\ell_{2}=0}^{\infty} M^{[m-1,r-\ell]}_{\ell_{2}}(k;a,b,c;\overline{\alpha}_{r-\ell})\frac{t^{\ell_{2}}}{\ell_{2}!}.
\end{eqnarray*}}
Using Cauchy product rule, we obtain \eqref{eq16}.
\end{proof}
\begin{theorem} The following relationship holds true
 \begin{equation}\label{eq17}
\sum_{k_{1}+k_{2}+...+k_{\ell}=n}\prod\limits_{i=1}^{\ell}\frac{ M^{[m-1,r_{i}]}_{k_{i}}(x_{i};k;a,b,c;\overline{\alpha}_{r_{i}})}{k_{1}!k_{2}!...k_{\ell}!}=\frac{1}{n!} M^{[m-1,\vert \mathbf{r} \vert]}_{n}(\vert \mathbf{x} \vert;k;a,b,c;\overline{\alpha}_{\vert \mathbf{r} \vert}), 
\end{equation}
where $\vert \mathbf{r} \vert=r_{1}+r_{2}+...r_{\ell}$ and $\vert \mathbf{x} \vert=x_{1}+x_{2}+...+x_{\ell}$ and       
  {\footnotesize $\overline{\alpha}_{r_{i}}=\left(\alpha_{\sum_{j=1}^{i-1}r_{j}},\alpha_{\sum_{j=1}^{i-1}r_{j}+1},..., \alpha_{\sum_{j=1}^{i}r_{j}-1}\right)$, $i=\{1,2,...,\ell\}$}.
\end{theorem}
 \begin{proof} Starting with \eqref{eq9}, we get
{\scriptsize \begin{multline*}
\sum\limits_{n=0}^{\infty}\left(M^{[m-1,\vert r \vert]}_{n}(\vert x \vert;k;a,b,c;\overline{\alpha}_{\vert r \vert})\right)\frac{t^{n}}{n!}= \frac{t^{\vert r \vert km}2^{\vert r \vert m(1-k)}c^{\vert x \vert t}}{\prod\limits_{i=0}^{\vert r \vert -1}\left(\alpha_{i}b^{t}-a^{t}\sum\limits_{\ell=0}^{m-1}\frac{t^{\ell}}{\ell!}\right)}\\
\quad\quad\quad\quad\quad\;\;\;=\frac{t^{r_{1}km}2^{r_{1}m(1-k)}c^{x_{1}t}}{\prod\limits_{i=0}^{r_{1}-1}\left(\alpha_{i}b^{t}-a^{t}\sum\limits_{\ell=0}^{m-1}\frac{t^{\ell}}{\ell!}\right)}\frac{t^{r_{2}km}2^{r_{2}m(1-k)}c^{x_{2}t}}{\prod\limits_{i=r_{1}}^{r_{1}+r_{2}-1}\left(\alpha_{i}b^{t}-a^{t}\sum\limits_{\ell=0}^{m-1}\frac{t^{\ell}}{\ell!}\right)}...\frac{t^{r_{\ell}km}2^{r_{\ell}m(1-k)}c^{x_{\ell}t}}{\prod\limits_{i=r_{1}+r_{2}...+r_{\ell-1}}^{r_{1}+r_{2}+...+r_{\ell}-1}\left(\alpha_{i}b^{t}-a^{t}\sum\limits_{\ell=0}^{m-1}\frac{t^{\ell}}{\ell!}\right)}\\
=\sum\limits_{k_{1}=0}^{\infty} M^{[m-1,r_{1}]}_{k_{1}}(x_{1};k;a,b,c;\overline{\alpha}_{r_{1}})\frac{t^{k_{1}}}{k_{1}!}\sum\limits_{k_{2}=0}^{\infty} M^{[m-1,r_{2}]}_{k_{2}}(x_{2};k;a,b,c;\overline{\alpha}_{r_{2}})\frac{t^{k_{2}}}{k_{2}!}...\sum\limits_{k_{\ell}=0}^{\infty} M^{[m-1,r_{\ell}]}_{k_{1}}(x_{\ell};k;a,b,c;\overline{\alpha}_{r_{\ell}})\frac{t^{k_{\ell}}}{k_{\ell}!}
\end{multline*}}
Using Cauchy product rule on the right hand side of the last equation and equating coefficients of $t^{n}$ on both sides, yields \eqref{eq17}.
\end{proof}
Using $\textbf{No.13}$ in $\textbf{Table1}$, we obtain \textit{N\"{o}rlun$d^{,}$s} results, see \cite{Norlund} and  \textit{Carlit$z^{,}$s} generalizations, see \cite{Carlitz} by our approach in Theorem 3.5 and Theorem 3.6 as follows

\begin{theorem} For $ (\overline{\alpha}_{r})^{n}=\left(\alpha_{0}^{n},\alpha_{1}^{n},...,\alpha_{r-1}^{n}\right)$, we have
{\footnotesize \begin{equation}\label{eq18}
\prod\limits_{i=1}^{r}\sum\limits_{s_{i}=0}^{n-1}(\alpha_{i-1})^{s_{i}} M^{[0,r]}_{\ell} \left(x+\frac{\sum\limits_{i=1}^{r}s_{i}}{n};k;1,e,e; (\overline{\alpha}_{r})^{n}\right)=n^{rk-\ell} M^{[0,r]}_{\ell} \left(nx+;k;1,e,e;\overline{\alpha}_{r}\right).
\end{equation}}
{\footnotesize \begin{equation}\label{eq19}
\prod\limits_{i=1}^{r}\sum\limits_{s_{i}=0}^{n-1}(\alpha_{i-1})^{s_{i}} M^{[0,r]}_{r+\ell} \left(x+\frac{\sum\limits_{i=0}^{r}s_{i}}{n};k;1,e,e; (\overline{\alpha}_{r})^{n}\right)=n^{r(k-1)-\ell}\frac{(\ell+r)!}{\ell!} M^{[0,r]}_{\ell} \left(nx+;k-1;1,e,e;\overline{\alpha}_{r}\right).
\end{equation}}
\end{theorem} 
\begin{proof} For the first equation and starting with \eqref{eq9}, we get
{\footnotesize \begin{eqnarray*}
\sum\limits_{\ell=0}^{\infty}\frac{{(nt)}^{\ell}}{\ell!}\prod\limits_{i=1}^{r}\sum\limits_{s_{i}=0}^{n-1}(\alpha_{i-1})^{s_{i}} M^{[0,r]}_{\ell} \left(x+\frac{\sum\limits_{i=1}^{r}s_{i}}{n};k;1,e,e; (\overline{\alpha}_{r})^{n}\right)
&=&\frac{(nt)^{rk}2^{r(1-k)}e^{nxt}}{\prod\limits_{i=0}^{r-1}\alpha_{i}^{n}e^{nt}-1}\prod\limits_{i=1}^{r}\sum\limits_{s_{i}=0}^{n-1}(\alpha_{i-1}e^{t})^{s_{i}} 
\end{eqnarray*}}
{ \footnotesize $$
\quad\quad\quad\quad\quad\quad\quad\quad\quad\quad\quad\quad\quad\quad\quad\quad\quad\quad\quad\quad=\frac{(nt)^{rk}2^{r(1-k)}e^{(nx)t}}{\prod\limits_{i=0}^{r-1}\alpha_{i}e^{t}-1}=n^{rk}\sum\limits_{\ell=0}^{\infty} M^{[0,r]}_{\ell} \left(nx;k;1,e,e;\overline{\alpha}_{r}\right)\frac{t^{\ell}}{\ell!}.
$$}
\\
Equating coefficients of $t^{\ell}$ on both sides, yields \eqref{eq18}.\\
For the second equation and starting with \eqref{eq9}, we get
{\footnotesize \begin{eqnarray*}
\sum\limits_{\ell=0}^{\infty}\frac{{(nt)}^{\ell}}{\ell!}\prod\limits_{i=1}^{r}\sum\limits_{s_{i}=0}^{n-1}(\alpha_{i-1})^{s_{i}} M^{[0,r]}_{\ell} \left(x+\frac{\sum\limits_{i=1}^{r}s_{i}}{n};k;1,e,e; (\overline{\alpha}_{r})^{n}\right)=\frac{n^{rk}t^{r}2^{-r}(t)^{r(k-1)}2^{r(2-k)}e^{nxt}}{\prod\limits_{i=0}^{r-1}\alpha_{i}^{n}e^{nt}-1} \prod\limits_{i=1}^{r}\sum\limits_{s_{i}=0}^{n-1}(\alpha_{i-1}e^{t})^{s_{i}}
\end{eqnarray*}}
$$
\quad\quad\quad\quad\quad\quad \quad\quad\quad\quad\quad\quad\quad\quad\quad\quad\quad\quad=\frac{n^{rk}t^{r}2^{-r}(t)^{r(k-1)}2^{r(2-k)}e^{nxt}}{\prod\limits_{i=0}^{r-1}\alpha_{i}e^{t}-1},
$$
\\
then, we have\\
{\footnotesize $
\sum\limits_{\ell=0}^{\infty}\frac{n^{\ell+r}\ell!}{(\ell+r)!}\frac{t^{\ell}}{\ell!}\prod\limits_{i=1}^{r}\sum\limits_{s_{i}=0}^{n-1}(\alpha_{i-1})^{s_{i}} M^{[0,r]}_{r+\ell} \left(x+\frac{\sum\limits_{i=1}^{r}s_{i}}{n};k;1,e,e; (\overline{\alpha}_{r})^{n}\right)=n^{rk}2^{-r}\sum\limits_{\ell=0}^{\infty} M^{[0,r]}_{\ell} \left(nx;k-1;1,e,e;\overline{\alpha}_{r}\right)\frac{t^{\ell}}{\ell!}.
$}
\\
Equating coefficients of $t^{\ell}$ on both sides, yields \eqref{eq19}.
\end{proof}

\begin{theorem} For $ ( \overline{\alpha}_{r})^{n} =\left(\alpha_{0}^{n},\alpha_{1}^{n},...,\alpha_{r-1}^{n}\right)$ and  $ (\overline{\alpha}_{r})^{m}=\left(\alpha_{0}^{m},\alpha_{1}^{m},...,\alpha_{r-1}^{m}\right)$ we have
{\footnotesize \begin{multline}\label{eq20}
n^{\ell}\prod\limits_{i=1}^{r}\sum\limits_{s_{i}=0}^{n-1}(\alpha_{i-1})^{ms_{i}} M^{[0,r]}_{\ell} \left(\frac{x}{n}+\frac{\left(\sum\limits_{i=1}^{r}s_{i}\right)m}{n};k;1,e,e; (\overline{\alpha}_{r})^{n}\right)= \\
m^{-rk+\ell}n^{rk}\prod\limits_{i=1}^{r}\sum\limits_{p_{i}=0}^{m-1}(\alpha_{i-1})^{np_{i}} M^{[0,r]}_{\ell} \left(\frac{x}{m}+\frac{\left(\sum\limits_{i=1}^{r}p_{i}\right)n}{m};k;1,e,e;(\overline{\alpha}_{r})^{m}\right).
\end{multline}}
{\footnotesize \begin{multline}\label{eq21}
n^{\ell+r}\prod\limits_{i=1}^{r}\sum\limits_{s_{i}=0}^{n-1}(\alpha_{i-1})^{ms_{i}} M^{[0,r]}_{\ell+r} \left(\frac{x}{n}+\frac{\left(\sum\limits_{i=1}^{r}s_{i}\right)m}{n};k;1,e,e; (\overline{\alpha}_{r})^{n}\right)=\\
\frac{m^{-r(k-1)+\ell}n^{rk}}{2^{r}}\prod\limits_{i=1}^{r}\sum\limits_{p_{i}=0}^{m-1}(\alpha_{i-1})^{np_{i}} M^{[0,r]}_{\ell} \left(\frac{x}{m}+\frac{\left(\sum\limits_{i=1}^{r}p_{i}\right)n}{m};k-1;1,e,e;(\overline{\alpha}_{r})^{m}\right).
\end{multline}}
\end{theorem} 
\begin{proof}  For the first equation and starting with \eqref{eq9}, we get
{\footnotesize \begin{eqnarray*}
\sum\limits_{\ell=0}^{\infty}\frac{{(nt)}^{\ell}}{\ell!}\prod\limits_{i=1}^{r}\sum\limits_{s_{i}=0}^{n-1}(\alpha_{i-1})^{ms_{i}} M^{[0,r]}_{\ell} \left(\frac{x}{n}+\frac{\left(\sum\limits_{i=1}^{r}s_{i}\right)m}{n};k;1,e,e; (\overline{\alpha}_{r})^{n}\right)=\frac{n^{rk}2^{r(1-k)}t^{rk}e^{xt}}{\prod\limits_{i=0}^{r-1}(\alpha_{i}^{n}e^{nt}-1)}\frac{\prod\limits_{i=0}^{r-1}(\alpha_{i}e^{nmt}-1)}{\prod\limits_{i=0}^{r-1}(\alpha_{i}^{m}e^{mt}-1)}
\end{eqnarray*}}
\\
 $$
=\frac{m^{-rk}n^{rk}2^{r(1-k)}t^{rk}m^{rk}e^{\left(\frac{x}{m}+\frac{\left(\sum\limits_{i=1}^{r}p_{i}\right)n}{m}\right)mt}}{\prod\limits_{i=0}^{r-1}(\alpha_{i}^{m}e^{mt}-1)}\prod\limits_{i=1}^{r}\sum\limits_{p_{i}=0}^{m-1}(\alpha_{i-1})^{np_{i}}
$$
\\
{ \footnotesize $$
\quad\quad\quad\quad\quad\quad=m^{-rk}n^{rk}\sum\limits_{\ell=0}^{\infty}\left(\prod\limits_{i=1}^{r}\sum\limits_{p_{i}=0}^{m-1}(\alpha_{i-1})^{np_{i}} m^{\ell} M^{[0,r]}_{\ell} \left(\frac{x}{m}+\frac{\left(\sum\limits_{i=1}^{r}p_{i}\right)n}{m};k;1,e,e;(\overline{\alpha}_{r})^{m}\right)\right)\frac{t^{\ell}}{\ell!}.
$$}
\\
Equating coefficients of $t^{\ell}$ on both sides, yields \eqref{eq20}.\\
Also, It is not difficult to prove \eqref{eq21}.
\end{proof}

\section{Some relations between the polynomials $ M^{[m-1,r]}_{n}(x;k;a,b,c;\overline{\alpha}_{r})$ and other polynomials and numbers}
In this section, we give some relationships between the polynomials $M^{[m-1,r]}_{n}(x;k;a,b,c;\overline{\alpha}_{r})$ and Laguerre polynomials, Jacobi polynomials, Hermite polynomials, generalized Stirling numbers of second kind, Stirling numbers and Bleimann-Butzer-hahn basic.
\begin{theorem} For $\overline{\alpha}_{r}=(\alpha_{0},\alpha_{1},...,\alpha_{r}) \in \mathbb{C}$, $(x;\overline{\alpha})_{\underline{\ell}}=(x-\alpha_{0})(x-\alpha_{1})...(x-\alpha_{\ell-1})$ and $n,j \in \mathbb{ N}_{0}$, we have relationship
\begin{equation}\label{eq22}
M^{[m-1,r]}_{n}(x;k;a,b,c;\overline{\alpha}_{r})=\sum\limits_{j=0}^{n}(x;\alpha)_{\underline{j}}\sum\limits_{\ell=j}^{n}\binom{n}{n-\ell}(\ln c)^{\ell}S(\ell,j;\overline{\alpha})M^{[m-1,r]}_{n-\ell}(k;a,b,c;\overline{\alpha}_{r})
\end{equation}
between the new unification of Apostol-type polynomials and generalized Stirling numbers of second kind, see \cite{Comtet}.
\end{theorem}
\begin{proof} Using \eqref{eq14} and from definition of generalized Stirling numbers of second kind, we easily obtain \eqref{eq22}.
\end{proof}
\begin{theorem} For $\overline{\alpha}_{r}=(\alpha_{0},\alpha_{1},...,\alpha_{r}) \in \mathbb{C}$, $(x)_{\underline{\ell}}=(x)(x-1)...(x-\ell+1)$ and $n,j \in \mathbb{ N}_{0}$, we have relationship
\begin{equation}\label{eq23}
M^{[m-1,r]}_{n}(x;k;a,b,c;\overline{\alpha}_{r})=\sum\limits_{j=0}^{n}(x)_{\underline{j}}\sum\limits_{\ell=j}^{n}\binom{n}{n-\ell}(\ln c)^{\ell}S(\ell,j) M^{[m-1,r]}_{n-\ell}(k;a,b,c;\overline{\alpha}_{r})
\end{equation}
between the new unification of Apostol-type polynomials and Stirling numbers of second kind.
\end{theorem}
\begin{proof} Using \eqref{eq14} and from definition of Stirling numbers of second kind, see \cite{Gould}, we easily obtain \eqref{eq23}.
\end{proof}
\begin{theorem} The relationship
{\footnotesize \begin{equation}\label{eq24}
M^{[m-1,r]}_{n}(x;k;a,b,c;\overline{\alpha}_{r})=\sum\limits_{j=0}^{n}\sum\limits_{\ell=j}^{n}(-1)^{j}\ell!\binom{n}{n-\ell}(\ln c)^{\ell}\binom{\ell+\alpha}{\ell-j}L^{(\alpha)}_{j}(x) M^{[m-1,r]}_{n-\ell}(k;a,b,c;\overline{\alpha}_{r})
\end{equation}}
holds between the new unification of multiparameter Apostol-type polynomials and generalized Laguerre polynomials, see \cite[No.(3) \textbf{Table1}]{Tremblay}.
\end{theorem}
\begin{proof} From \eqref{eq14} and substitute 
\begin{equation*}
x^{\ell}=\ell!\sum\limits_{j=0}^{\ell}(-1)^{j}\binom{\ell+\alpha}{\ell-j}L^{\alpha}_{j}(x), 
\end{equation*}
then we get \eqref{eq24}.
\end{proof}
\begin{theorem} For $ (\alpha+\beta+j+1)_{\ell+1}=(\alpha+\beta+j+1)(\alpha+\beta+j+2)...(\alpha+\beta+j+\ell+1)$. The relationship
\begin{multline}\label{eq25}
M^{[m-1,r]}_{n}(x;k;a,b,c;\overline{\alpha}_{r})=\sum\limits_{j=0}^{n}\sum\limits_{\ell=j}^{n}(-1)^{j}\ell!\binom{n}{n-\ell}(\ln c)^{\ell}\binom{\ell+\alpha}{\ell-j}\frac{\alpha+\beta+2j+1}{(\alpha+\beta+j+1)_{\ell+1}}\\
P^{(\alpha,\beta)}_{j}(1-2x) M^{[m-1,r]}_{n-\ell}(k;a,b,c;\overline{\alpha}_{r})
\end{multline}
holds between the new unification of Apostol-type polynomials and Jacobi polynomials, see\cite[ p.49, Eq. (35)]{Sriva4}.
\end{theorem}
\begin{proof} From \eqref{eq14} and substitute 
\begin{equation*}
x^{\ell}=\ell!\sum\limits_{j=0}^{\ell}(-1)^{j}\binom{\ell+\alpha}{\ell-j}\frac{\alpha+\beta+2j+1}{(\alpha+\beta+j+1)_{\ell+1}}P^{(\alpha,\beta)}_{j}(1-2x), 
\end{equation*}
then we get \eqref{eq25}.
\end{proof}
\begin{theorem} The relationship
{\footnotesize \begin{equation}\label{eq26}
M^{[m-1,r]}_{n}(x;k;a,b,c;\overline{\alpha}_{r})=\sum\limits_{j=0}^{[\frac{n}{2}]}\sum\limits_{\ell=2j}^{n}2^{-\ell}\binom{n}{n-\ell}\binom{\ell}{2j}\frac{2j!}{j!}(\ln c)^{\ell}H_{\ell-2j}(x) M^{[m-1,r]}_{n-\ell}(k;a,b,c;\overline{\alpha}_{r})
\end{equation}}
holds between the new unification of Apostol-type polynomials and Hermite polynomials, see \cite[No.(1) \textbf{Table1}]{Tremblay}.
\end{theorem}
\begin{proof} From \eqref{eq14} and substitute 
\begin{equation*}
x^{\ell}=2^{-\ell}\sum\limits_{j=0}^{[\frac{\ell}{2}]}\binom{\ell}{2j}\frac{2j!}{j!}H_{\ell-2j}(x), 
\end{equation*}
then we get \eqref{eq26}.
\end{proof}
\begin{theorem} When $m=1$, $a=1$, $b=e$ and $c=e$ in \eqref{eq9} and for $\overline{\alpha}_{r}=\left(\alpha_{0},\alpha_{1},...,\alpha_{r-1}\right)$, $\overline{\alpha^{*}}_{r}=\left(\frac{1}{\alpha_{0}},\frac{1}{\alpha_{1}},...,\frac{1}{\alpha_{r-1}}\right)$, $\alpha_{i} \neq 0$, $i=0,1,...,r-1$ and $\overline{\beta}_{m}=\left(\beta_{0},\beta_{1},...,\beta_{m-1}\right), \overline{\beta^{*}}_{m}=\left(\frac{1}{\beta_{0}},\frac{1}{\beta_{1}},...,\frac{1}{\beta_{m-1}}\right)$, $\beta_{i} \neq 0$, $i=0,1,...,m-1$, we have the following relationship  
\begin{equation}\label{eq9999}
M_{n}^{(r)}(x;k;\overline{\alpha}_{r})=\frac{n!}{\prod_{i=0}^{r-1}\alpha_{i}}\sum_{m=r}^{\infty}\frac{2^{(1-k)(r-m)}\prod_{j=0}^{m-1}\beta_{j}}{(n+k(m-1))!}C(m,r;\overline{\alpha^{*}}_{r};\overline{\beta^{*}}_{m})M_{n+k(m-r)}^{(m)}(x;k;\overline{\beta}_{m}),
\end{equation}
between the new unified family of generalized Apostol-Euler, Bernoulli and Genocchi polynomials, and $C(m,r;\overline{\alpha^{*}}_{r};\overline{\beta^{*}}_{m})$ (the generalized Lah numbers), see \cite{Cha}. 
\end{theorem}
\begin{proof} From \cite[Eq. 2.1]{Desouky2},
{\footnotesize \begin{eqnarray*}
\sum_{n=0}^{\infty}M_{n}^{(r)}(x;k;\overline{\alpha}_{r})\frac{t^{n}}{n!}&=&\frac{t^{rk}2^{r(1-k)}e^{xt}}{\prod_{i=1}^{r-1}(\alpha_{i}e^{t}-1})\\
 &=&\frac{t^{rk}2^{r(1-k)}e^{xt}}{\prod_{i=1}^{r-1}\alpha_{i}}\frac{1}{(e^{t};\overline{\alpha^{*}})_{\underline{r}}}\\
 &=&\frac{t^{rk}2^{r(1-k)}e^{xt}}{\prod_{i=1}^{r-1}\alpha_{i}}\sum_{m=r}^{\infty}C(m,r;\overline{\alpha^{*}}_{r};\overline{\beta^{*}}_{m})\frac{1}{(e^{t};\overline{\beta^{*}})_{\underline{m}}}\\
&=&\sum_{m=r}^{\infty}\frac{t^{k(r-m)}2^{(r-m)(1-k)}\prod_{j=1}^{m-1}\beta_{j}}{\prod_{i=1}^{r-1}\alpha_{i}}C(m,r;\overline{\alpha^{*}}_{r};\overline{\beta^{*}}_{m})\frac{t^{mk}2^{m(1-k)}e^{xt}}{(e^{t};\overline{\beta^{*}})_{\underline{m}}}\\
&=&\sum_{n=0}^{\infty}\left(\sum_{m=r}^{\infty}\frac{ n!2^{(1-k)(r-m)}\prod_{j=0}^{m-1}\beta_{j}}{(n+k(m-1))!\prod_{i=0}^{r-1}\alpha_{i}}C(m,r;\overline{\alpha^{*}}_{r};\overline{\beta^{*}}_{m})M_{n+k(m-r)}^{(m)}(x;k;\overline{\beta}_{m})\right)\frac{t^{n}}{n!}.
\end{eqnarray*}}
Equating the coefficients of $t^{n}$ on both sides, yields \eqref{eq9999}.
\end{proof}
Using $\textbf{No.13}$ in $\textbf{Table 1}$, see \cite{Desouky2}  and the definition of  the unified Bernstein and Bleimann-Butzer-Hahn basis(see \cite{Bozer}),
\begin{equation}\label{eq27}
\left(\frac{2^{1-k}x^{k}t^{k}}{(1+ax)^{k}}\right)^{m}\frac{1}{mk!}e^{t\left(\frac{1+bx}{1+ax}\right)}=\sum\limits_{n=0}^{\infty}p^{(a,b)}_{n}(x;k,m)\frac{t^{n}}{n!},
\end{equation}
where $k,m \in \mathbb{Z^{+}}$, $ a,b \in \mathbb{R}$, $t \in \mathbb{C} $, we obtain the following theorem
\begin{theorem} For $ \alpha_{i} \neq 0, i=0,1,...,r-1,$ we have relationship 
\begin{equation}\label{eq28}
P^{(a,b)}_{n}(x;k,r)=\frac{\prod\limits_{i=0}^{r-1}\alpha_{i}}{rk!}\left(\frac{x}{1+ax}\right)^{rk}\sum\limits_{j=0}^{r}s\left(r,j;\frac{1}{\overline{\alpha}_{r}}\right)\sum\limits_{\ell=0}^{n}j^{n-\ell}\binom{n}{\ell}M^{(r)}_{\ell}\left(\frac{1+bx}{1+ax};k;\overline{\alpha}_{r}\right)
\end{equation}
between the unified Bernstein and Bleimann-Butzer-Hahn basis, the new unified family of generalized Apostol-Bernoulli, Euler and Genocchi polynomials, see \cite{Desouky2} and generalized Stirling numbers of first kind, see \cite{Comtet}.
\end{theorem}
\begin{proof} From \eqref{eq9} and \eqref{eq27} and with some elementary calculation, we easily obtain \eqref{eq28}.
\end{proof}
$\mathbf{References}$
\bibliographystyle{elsarticle-num}

\end{document}